\documentclass[11pt,final]{article}

\usepackage{amsmath}
\usepackage{amsthm}
\usepackage{amssymb}

\numberwithin{equation}{section}

\theoremstyle{definition}\newtheorem{definition}{Definition}
\theoremstyle{plain}\newtheorem{theorem}[definition]{Theorem}
\theoremstyle{plain}
\theoremstyle{plain}
\theoremstyle{plain}\newtheorem{lemma}[definition]{Lemma}
\theoremstyle{definition}\newtheorem{assumption}[definition]{Assumption}
\theoremstyle{definition}
\theoremstyle{definition}
\theoremstyle{definition}\newtheorem{remark}[definition]{Remark}

\newcommand{\cO}{{\mathcal O}}

\newcommand{\bbR}{\mathbb{R}}

\newcommand{\la}{\langle}
\newcommand{\ra}{\rangle}

\begin{document}

\title{A converse result for Banach space convergence rates in Tikhonov-type convex regularization of ill-posed linear equations}

\author{
\textsc{Jens Flemming}%
\footnote{Chemnitz University of Technology,
Faculty of Mathematics, D-09107 Chemnitz, Germany,
jens.flemming@mathematik.tu-chemnitz.de.}
}

\date{\today
}

\maketitle

\begin{abstract}
We consider Tikhonov-type variational regularization of ill-posed linear operator equations in Banach spaces with general convex penalty functionals.
Upper bounds for certain error measures expressing the distance between exact and regularized solutions, especially for Bregman distances, can be obtained from variational source conditions.
We prove that such bounds are optimal in case of twisted Bregman distances, that is, the rate function is also an asymptotic lower bound for the error measure. This result extends existing converse results from Hilbert space settings to Banach spaces without adhering to spectral theory.
\end{abstract}

\section{Setting}

We look at ill-posed linear operator equations
\begin{equation}\label{eq:the}
A\,x=y^\dagger,\quad x\in X,
\end{equation}
in Banach spaces $X$ and $Y$, where $A:X\rightarrow Y$ is a bounded linear operator and the exact right-hand side $y^\dagger\in Y$ is only accessible through a noisy measurment $y^\delta\in Y$ satisfying
\begin{equation*}
\|y^\delta-y^\dagger\|\leq\delta
\end{equation*}
with noise level $\delta\geq 0$.
Here, ill-posedness refers to a discontinuous dependence of the solutions on the data, in essence.

Regularization is required and we restrict our attention to Tikhonov-type regulariztion of the form
\begin{equation}\label{eq:tikh}
T_\alpha^\delta(x):=\frac{1}{p}\,\|A\,x-y^\delta\|^p+\alpha\,\Omega(x)\to\min_{x\in X},
\end{equation}
which is quite common in Banach space regularization, see \cite[Section~3.2]{SchGraGroHalLen09} or \cite[Chapter~4]{SchKalHofKaz12} and references therein. The functional $\Omega:X\rightarrow(-\infty,+\infty]$ shall stabilize the minimization problem and the regularization parameter $\alpha>0$ controls the trade-off between data fitting and stabilization. The exponent $p>1$ can be chosen to ease numerical minimization, e.\,g., $p=2$ in Hilbert spaces.
Existence, stability and convergence of the minimizers are guaranteed by the following assumptions, see \cite[Section~3.2]{SchGraGroHalLen09} or \cite[Chapter~4]{SchKalHofKaz12}.

\begin{assumption}
We assume that the following properties are satisfied by the introduced setting.
\begin{itemize}
\item[(i)]
Equation~\eqref{eq:the} has a solution with finite $\Omega$.
\item[(ii)]
$\Omega$ is proper and convex.
\item[(iii)]
The sublevel sets $\{x\in X:\,\Omega(x)\leq c\}$, $c\in\bbR$, are weakly sequentially closed and each sequence in such a set has a weakly convergent subsequence.
\end{itemize}
\end{assumption}

\section{Convergence rates}\label{sc:rates}

Tikhonov regularized solutions converge weakly, at least in a subsequential manner, to solutions which minimize the penalty $\Omega$ in the set of all solutions to \eqref{eq:the}. By $x^\dagger$ we typically denote an $\Omega$ minimizing solution.
Convergence to $\Omega$ minimizing solutions can be arbitrarily slow and additional effort is needed to estimate the speed of convergence, at least asymptotically as $\delta$ goes to zero and $\alpha=\alpha(\delta,y^\delta)$ is chosen properly. To formulate such convergence rate results we first have to choose an error measure expressing the distance between regularized and exact solutions and then we have to look for suitable conditions guaranteeing a certain convergence speed.

There is a wide choice of error measures. In Hilbert spaces the standard choice is the squared norm
\begin{equation*}
\|x_\alpha^\delta-x^\dagger\|^2,
\end{equation*}
where $x_\alpha^\delta$ is a minimizer of \eqref{eq:tikh} and $x^\dagger$ is an $\Omega$ minimizing solution.
In Banach spaces a typical choice are Bregman distances with respect to the convex penalty $\Omega$.

\begin{definition}
Let $\bar{x}\in X$ be an element with nonempty subdifferential $\partial\Omega(\bar{x})$ and let $\bar{\xi}\in\partial\Omega(\bar{x})$ be a subgradient of $\Omega$ at $\bar{x}$. The functional \mbox{$B^\Omega_{\bar{\xi}}(\cdot,\bar{x}):X\rightarrow[0,+\infty]$} defined by
\begin{equation*}
B^\Omega_{\bar{\xi}}(x,\bar{x}):=\Omega(x)-\Omega(\bar{x})-\la\bar{\xi},x-\bar{x}\ra_{X^\ast\times X},\qquad x\in X,
\end{equation*}
is called \emph{Bregman distance} between $x$ and $\bar{x}$.
\end{definition}

Usually one uses
\begin{equation*}
B^\Omega_{\xi^\dagger}(x_\alpha^\delta,x^\dagger)
\end{equation*}
as error measure, where $\xi^\dagger\in\partial\Omega(x^\dagger)$. Note that $\partial\Omega(x^\dagger)\neq\emptyset$ is not automatically fulfilled and has to be seen as an additional assumption.
In \cite{Kin16} the skewed Bregman distance
\begin{equation*}
B^\Omega_{\xi_\alpha^\delta}(x^\dagger,x_\alpha^\delta)
\end{equation*}
has been suggested as an error measure for convergence rate results, where $\xi_\alpha^\delta\in\partial\Omega(x_\alpha^\delta)$.
From the following lemma we immediately see that \mbox{$\partial\Omega(x_\alpha^\delta)\neq\emptyset$} is automatically fulfilled.

\begin{lemma}\label{th:tikhopt}
An element $x_\alpha^\delta\in X$ is a minimizer of \eqref{eq:tikh} if and only if there is some $\eta_\alpha^\delta\in Y^\ast$ such that
\begin{equation*}
\eta_\alpha^\delta\in-\frac{1}{\alpha}\,\partial\left(\frac{1}{p}\,\|\cdot\|^p\right)(A\,x_\alpha^\delta-y^\delta)
\quad\text{and}\quad A^\ast\,\eta_\alpha^\delta\in\partial\Omega(x_\alpha^\delta).
\end{equation*}
\end{lemma}

\begin{proof}
The proof is elementary convex analysis, see \cite[Corollary~2.25]{SchKalHofKaz12}.
\end{proof}

We use the abbreviation $\xi_\alpha^\delta:=A^\ast\,\eta_\alpha^\delta$ with $\eta_\alpha^\delta$ from the lemma throughout the article.
Note that
\begin{equation}\label{eq:normsub}
\partial\left(\frac{1}{p}\,\|\cdot\|^p\right)(y)=\bigl\{\eta\in Y^\ast:\la\eta,y\ra_{Y^\ast\times Y}=\|\eta\|\,\|y\|,\|\eta\|=\|y\|^{p-1}\bigr\}
\end{equation}
holds for all $y\in Y$, see \cite[Section~2.2.2]{SchKalHofKaz12}.

Both Bregman distances, the usual one and the skewed one, specialize to the squared norm distance in Hilbert spaces if $\Omega=\|\cdot\|^2$.
The advantage of the skewed version is that we do not need additional assumptions to guarantee its existence in case of Tikhonov regularization for linear equations~\eqref{eq:the}. If the operator $A$ would be replaced by a nonlinear mapping, then it is not clear whether each Tikhonov minimizer has a nonempty sub\-differential. Thus, the skewed Bregman distance might not be available, but the usual Bregman distance still works if there is an $\Omega$ minimizing solution with nonempty subdifferential.

We aim at asymptotic convergence speed estimates (convergence rates) of the form
\begin{equation*}
B^\Omega_{\xi_\alpha^\delta}(x^\dagger,x_\alpha^\delta)=\cO(\varphi(\delta)),\qquad\delta\to 0,
\end{equation*}
for some parameter choice $\alpha=\alpha(\delta,y^\delta)$.
The rate function $\varphi$ usually is an index function.

\begin{definition}
A function $\varphi:[0,\infty)\rightarrow[0,\infty)$ is an \emph{index function} if it is continuous, monotonically increasing, strictly increasing in a neighborhood of zero, and satisfies $\varphi(0)=0$.
\end{definition}

\begin{remark}
Simple calculations show that for \emph{concave} index functions $\varphi$ we have that
$t\mapsto\frac{\varphi(t)}{t}$, $t>0$, is monotonically decreasing and, as a consequence, that
$\varphi(c\,t)\leq c\,\varphi(t)$ for all $t\geq 0$ if $c\geq 1$.
Both properties will be used in subsequent proofs without further notice.
\end{remark}

The following theorem appeared in \cite[Theorem~3.1]{Kin16} for general non-norm fitting functionals, with a different a priori parameter choice and restricted to concave index functions $\varphi$.
The proof provided below is a simplified version of the original proof. This simplification stems from restricting our attention to Banach space norms as fitting functionals and from the simpler parameter choice rule we borrowed from \cite[Theorem~1]{HofMat12}.

\begin{theorem}\label{th:rate}
Let $x^\dagger$ be an $\Omega$ minimizing solution to \eqref{eq:the} and assume
\begin{equation}\label{eq:vsc}
0\leq\Omega(x)-\Omega(x^\dagger)+\varphi(\|A\,x-A\,x^\dagger\|)\qquad\text{for all $x\in X$}
\end{equation}
with an index function $\varphi$ for which $t\mapsto\frac{\varphi(t)}{t}$, $t>0$, is monotonically decreasing (e.\,g., $\varphi$ concave).
Then
\begin{equation*}
B^\Omega_{\xi_\alpha^\delta}(x^\dagger,x_\alpha^\delta)=\cO(\varphi(\delta)),\qquad\delta\to 0,
\end{equation*}
where $\alpha=\alpha(\delta)$ is chosen a priori such that
\begin{equation*}
c_1\,\frac{\delta^p}{\varphi(\delta)}\leq\alpha(\delta)\leq c_2\,\frac{\delta^p}{\varphi(\delta)}
\end{equation*}
with positive constants $c_1$, $c_2$.
\end{theorem}

\begin{proof}
Let $\eta_\alpha^\delta$ be as in Lemma~\ref{th:tikhopt} and set $\xi_\alpha^\delta:=A^\ast\,\eta_\alpha^\delta$.
Taking into account \eqref{eq:normsub} we have
\begin{align}
B^\Omega_{\xi_\alpha^\delta}(x^\dagger,x_\alpha^\delta)
&=\Omega(x^\dagger)-\Omega(x_\alpha^\delta)-\la\xi_\alpha^\delta,x^\dagger-x_\alpha^\delta\ra_{X^\ast\times X}\nonumber\\
&=\Omega(x^\dagger)-\Omega(x_\alpha^\delta)+\la\eta_\alpha^\delta,A\,x_\alpha^\delta-y^\delta\ra_{Y^\ast\times Y}+\la\eta_\alpha^\delta,y^\delta-y^\dagger\ra_{Y^\ast\times Y}\nonumber\\
&\leq\Omega(x^\dagger)-\Omega(x_\alpha^\delta)+\la\eta_\alpha^\delta,A\,x_\alpha^\delta-y^\delta\ra_{Y^\ast\times Y}+\|\eta_\alpha^\delta\|\,\delta\nonumber\\
&=\Omega(x^\dagger)-\Omega(x_\alpha^\delta)-\frac{1}{\alpha}\,\|A\,x_\alpha^\delta-y^\delta\|^p+\frac{\delta}{\alpha}\,\|A\,x_\alpha^\delta-y^\delta\|^{p-1}\label{eq:proof1}
\end{align}
If $\|A\,x_\alpha^\delta-y^\delta\|\leq\delta$, then \eqref{eq:proof1} and \eqref{eq:vsc} imply
\begin{align*}
B^\Omega_{\xi_\alpha^\delta}(x^\dagger,x_\alpha^\delta)
&\leq\varphi(\|A\,x_\alpha^\delta-A\,x^\dagger\|)+\frac{\delta}{\alpha}\,\|A\,x_\alpha^\delta-y^\delta\|^{p-1}\\
&\leq\varphi(\|A\,x_\alpha^\delta-y^\delta\|+\delta)+\frac{\delta^p}{\alpha}\\
&\leq\varphi(2\,\delta)+\frac{1}{c_1}\,\varphi(\delta)\\
&\leq\left(2+\frac{1}{c_1}\right)\,\varphi(\delta).
\end{align*}
If, on the other hand, $\|A\,x_\alpha^\delta-y^\delta\|>\delta$, then \eqref{eq:proof1} and $B^\Omega_{\xi_\alpha^\delta}(x^\dagger,x_\alpha^\delta)\geq 0$ imply
\begin{align*}
\|A\,x_\alpha^\delta-y^\delta\|^p
&\leq\alpha\,\Omega(x^\dagger)-\alpha\,\Omega(x_\alpha^\delta)+\|A\,x_\alpha^\delta-y^\delta\|^p\\
&=(1+p)\,\alpha\,(\Omega(x^\dagger)-\Omega(x_\alpha^\delta))+p\,T_\alpha^\delta(x_\alpha^\delta)-p\,\alpha\,\Omega(x^\dagger)\\
&\leq(1+p)\,\alpha\,(\Omega(x^\dagger)-\Omega(x_\alpha^\delta))+\delta^p,
\end{align*}
where we used $T_\alpha^\delta(x_\alpha^\delta)\leq T_\alpha^\delta(x^\dagger)$ in the last line.
From \eqref{eq:vsc} we now obtain
\begin{align*}
\|A\,x_\alpha^\delta-y^\delta\|^p
&\leq(1+p)\,\alpha\,\varphi(\|A\,x_\alpha^\delta-A\,x^\dagger\|)+\delta^p\\
&\leq(1+p)\,\alpha\,\frac{\varphi(\|A\,x_\alpha^\delta-y^\delta\|+\delta)}{\|A\,x_\alpha^\delta-y^\delta\|+\delta}\,(\|A\,x_\alpha^\delta-y^\delta\|+\delta)+\delta^p\\
&\leq(1+p)\,\alpha\,\frac{\varphi(\delta)}{\delta}\cdot 2\,\|A\,x_\alpha^\delta-y^\delta\|+\delta^{p-1}\,\|A\,x_\alpha^\delta-y^\delta\|.
\end{align*}
Thus,
\begin{equation*}
\|A\,x_\alpha^\delta-y^\delta\|^{p-1}
\leq 2\,(1+p)\,\alpha\,\frac{\varphi(\delta)}{\delta}+\delta^{p-1}
\leq(2\,(1+p)\,c_2+1)\,\delta^{p-1},
\end{equation*}
yielding
\begin{equation*}
\|A\,x_\alpha^\delta-y^\delta\|
\leq(2\,(1+p)\,c_2+1)^{\frac{1}{p-1}}\,\delta.
\end{equation*}
Applying \eqref{eq:vsc} and this estimate to \eqref{eq:proof1} we obtain
\begin{align*}
B^\Omega_{\xi_\alpha^\delta}(x^\dagger,x_\alpha^\delta)
&\leq\varphi(\|A\,x_\alpha^\delta-A\,x^\dagger\|)+\frac{\delta}{\alpha}\,\|A\,x_\alpha^\delta-y^\delta\|^{p-1}\\
&\leq\varphi(\|A\,x_\alpha^\delta-y^\delta\|+\delta)+(2\,(1+p)\,c_2+1)\,\frac{\delta^p}{\alpha}\\
&\leq\varphi\left(\left((2\,(1+p)\,c_2+1)^{\frac{1}{p-1}}+1\right)\,\delta\right)+\frac{2\,(1+p)\,c_2+1}{c_1}\,\varphi(\delta)\\
&\leq\left((2\,(1+p)\,c_2+1)^{\frac{1}{p-1}}+1+\frac{2\,(1+p)\,c_2+1}{c_1}\right)\,\varphi(\delta).
\end{align*}
This completes the case $\|A\,x_\alpha^\delta-y^\delta\|>\delta$ and the proof.
\end{proof}

In the proof we applied inequality \eqref{eq:vsc} to $x=x_\alpha^\delta$.
In the original proof it is applied to $x=x_\alpha^0$ only, which in some sense is more appealing. On the other hand, inequalities of the type \eqref{eq:vsc} usually hold on the hole space $X$.

Inequalities \eqref{eq:vsc} were introduced to convergence rate theory in \cite{HofKalPoeSch07} and are known as variational inequalities or variational source conditions.
Typically they have some error measure like norms or Bregman distances on the left-hand side instead of zero. But, as already known for linear $\varphi$ (see \cite[Proposition~12.25]{Fle12}) and shown in \cite{Kin16} for general $\varphi$, this is not mandatory for proving convergence rates.

Variational source conditions can be obtained from classical and general source conditions for linear problems in Hilbert spaces (see \cite[Section~3.2]{EngHanNeu96}, \cite{MatPer03}, \cite[Chapter~13]{Fle12}), as well as from problem specific calculations for certain nonlinear problems and Banach space settings (see \cite{BueFleHof16,FleGer17,HofKalPoeSch07,HohWei15,KoeHohWer16}).
Details on variational source conditions can be found in, e.\,g., \cite{BotHof10,Fle12,Gra10b,HohWer13}.

\begin{remark}
From \cite[Theorem~3.2]{Fle17} we know that variational source conditions are almost always available.
In particular, there is always a concave index function $\varphi$ such that \eqref{eq:vsc} holds.
\end{remark}

\section{Converse result}

We want to show that the convergence rate for the skewed Bregman distance in Theorem~\ref{th:rate} based on the variational source condition \eqref{eq:vsc} is optimal, that is, that there is a positive constant $c$ with
\begin{equation*}
c\,\varphi(\delta)\leq B^\Omega_{\xi_\alpha^\delta}(x^\dagger,x_\alpha^\delta)\qquad\text{for all $\delta>0$},
\end{equation*}
where $\alpha=\alpha(\delta)$ is chosen as in Theorem~\ref{th:rate}.
To achieve this aim we have to find the best possible index function $\varphi$ in \eqref{eq:vsc}.
For this purpose we apply the idea of approximate variational source conditions studied in \cite[Chapter~12]{Fle12}.

Define the function $D:[0,\infty)\rightarrow[0,\infty)$ by
\begin{equation}\label{eq:D}
D(r):=\sup_{x\in X}\left(\Omega(x^\dagger)-\Omega(x)-r\,\|A\,x-A\,x^\dagger\|\right),\qquad r\geq 0.
\end{equation}
This function is easily seen to be continuous, monotonically decreasing and convex.
Further, from \cite[proof of Theorem~3.2]{Fle17} we see
\begin{equation*}
\lim_{r\to\infty}D(r)=0.
\end{equation*}
The benchmark variational source condition
\begin{equation*}
0\leq\Omega(x)-\Omega(x^\dagger)+c\,\|A\,x-A\,x^\dagger\|\qquad\text{for all $x\in X$}
\end{equation*}
with some $c>0$ is satisfied if and only if $D(r)=0$ for some $r$.
Else we have $D(r)>0$ for all $r\geq 0$.

For later reference we define the auxiliary function $\Phi:(0,\infty)\rightarrow[0,\infty)$ by
\begin{equation}\label{eq:Phi}
\Phi(r):=\frac{D(r)}{r}.
\end{equation}
This function has a well-defined inverse $\Phi^{-1}:(0,\infty)\rightarrow(0,\infty)$ if $D(r)>0$ for all $r\geq 0$.

\begin{lemma}\label{th:Dvsc}
If $D(r)>0$ for all $r\geq 0$, then the variational source condition \eqref{eq:vsc} holds with
\begin{equation*}
\varphi(t)=\begin{cases}2\,D\left(\Phi^{-1}(t)\right),& t>0,\\0,&t=0,\end{cases}
\end{equation*}
and this function $\varphi$ is an index function for which $t\mapsto\frac{\varphi(t)}{t}$, $t>0$, is monotonically decreasing.
\end{lemma}

\begin{proof}
From the definition \eqref{eq:D} of $D$ we obtain
\begin{equation*}
0\leq\Omega(x)-\Omega(x^\dagger)+r\,\|A\,x-A\,x^\dagger\|+D(r)\qquad\text{for all $r\geq 0$.}
\end{equation*}
With $r=\Phi^{-1}(\|A\,x-A\,x^\dagger\|)$, that is $\|A\,x-A\,x^\dagger\|=\frac{D(r)}{r}$, this becomes
\begin{equation*}
0\leq\Omega(x)-\Omega(x^\dagger)+2\,D\bigl(\Phi^{-1}(\|A\,x-A\,x^\dagger\|)\bigr).
\end{equation*}
\par
That $\varphi$ is an index function can be easily deduced from the definition of $\Phi$ and from the properties of $D$.
From
\begin{equation*}
\frac{\varphi(t)}{t}=\frac{D(\Phi^{-1}(t))}{t}=\frac{D(\Phi^{-1}(t))}{\Phi^{-1}(t)}\,\frac{\Phi^{-1}(t)}{t}=\Phi(\Phi^{-1}(t))\,\frac{\Phi^{-1}(t)}{t}=\Phi^{-1}(t)
\end{equation*}
we see that $t\mapsto\frac{\varphi(t)}{t}$ is decreasing.
\end{proof}

\begin{theorem}\label{th:converse}
Let $D$ and $\Phi$ be defined by \eqref{eq:D} and \eqref{eq:Phi}.
If $D(r)>0$ for all $r\geq 0$, then there is a constant $c>0$ such that
\begin{equation*}
D\left(c\,\Phi^{-1}(\delta)\right)\leq B^\Omega_{\xi_\alpha^\delta}(x^\dagger,x_\alpha^\delta)\qquad\text{for all $\delta>0$},
\end{equation*}
where $\alpha=\alpha(\delta)$ is chosen a priori such that
\begin{equation*}
c_1\,\frac{\delta^p}{D(\Phi^{-1}(\delta))}\leq\alpha(\delta)\leq c_2\,\frac{\delta^p}{D(\Phi^{-1}(\delta))}
\end{equation*}
with positive constants $c_1$, $c_2$.
\end{theorem}

\begin{proof}
With $\eta_\alpha^\delta$ as in Lemma~\ref{th:tikhopt} we have
\begin{align*}
D(\|\eta_\alpha^\delta\|)
&=\sup_{x\in X}\bigl(\Omega(x^\dagger)-\Omega(x)-\|\eta_\alpha^\delta\|\,\|A\,x-A\,x^\dagger\|\bigr)\\
&\leq\sup_{x\in X}\bigl(\Omega(x^\dagger)-\Omega(x)+\la\eta_\alpha^\delta,A\,x-A\,x^\dagger\ra_{Y^\ast\times Y}\bigr)\\
&=\sup_{x\in X}\bigl(\Omega(x^\dagger)-\Omega(x)-\la A^\ast\,\eta_\alpha^\delta,x^\dagger-x\ra_{X^\ast\times X}\bigr)
\end{align*}
and the supremum is attained at $x_\alpha^\delta$, because $A^\ast\,\eta_\alpha^\delta\in\partial\Omega(x_\alpha^\delta)$.
Thus,
\begin{equation*}
D(\|\eta_\alpha^\delta\|)
\leq\Omega(x^\dagger)-\Omega(x_\alpha^\delta)-\la A^\ast\,\eta_\alpha^\delta,x^\dagger-x_\alpha^\delta\ra_{X^\ast\times X}
=B^\Omega_{\xi_\alpha^\delta}(x^\dagger,x_\alpha^\delta).
\end{equation*}
\par
It remains to show $\|\eta_\alpha^\delta\|\leq c\,\Phi^{-1}(\delta)$. Then the monotonicity of $D$ implies the asserted lower bound for the Bregman distance.
\par
If $\|A\,x_\alpha^\delta-y^\delta\|\leq\delta$, then by \eqref{eq:normsub} we have
\begin{equation*}
\|\eta_\alpha^\delta\|
=\frac{1}{\alpha}\,\|A\,x_\alpha^\delta-y^\delta\|^{p-1}
\leq\frac{\delta^{p-1}}{\alpha}
\leq\frac{1}{c_1}\,\frac{D(\Phi^{-1}(\delta))}{\delta}
=\frac{1}{c_1}\,\Phi^{-1}(\delta).
\end{equation*}
If, on the other hand, $\|A\,x_\alpha^\delta-y^\delta\|>\delta$, then by the minimizing property of $x_\alpha^\delta$ and by Lemma~\ref{th:Dvsc} we have
\begin{align*}
\frac{1}{p}\,\|A\,x_\alpha^\delta-y^\delta\|^p
&\leq\frac{1}{p}\,\delta^p+\alpha\,\Omega(x^\dagger)-\alpha\,\Omega(x_\alpha^\delta)\\
&\leq\frac{1}{p}\,\delta^p+2\,\alpha\,D\left(\Phi^{-1}(\|A\,x_\alpha^\delta-A\,x^\dagger\|)\right)\\
&\leq\frac{1}{p}\,\delta^p+2\,\alpha\frac{D\left(\Phi^{-1}(\|A\,x_\alpha^\delta-y^\delta\|+\delta)\right)}{\|A\,x_\alpha^\delta-y^\delta\|+\delta}\,(\|A\,x_\alpha^\delta-y^\delta\|+\delta)\\
&\leq\frac{1}{p}\,\delta^{p-1}\,\|A\,x_\alpha^\delta-y^\delta\|+2\,\alpha\,\frac{D\left(\Phi^{-1}(\delta)\right)}{\delta}\cdot 2\,\|A\,x_\alpha^\delta-y^\delta\|.
\end{align*}
Consequently,
\begin{equation*}
\|A\,x_\alpha^\delta-y^\delta\|^{p-1}
\leq\delta^{p-1}+4\,p\,\alpha\,\frac{D\left(\Phi^{-1}(\delta)\right)}{\delta}
\end{equation*}
and therefore
\begin{align*}
\|\eta_\alpha^\delta\|
&=\frac{1}{\alpha}\,\|A\,x_\alpha^\delta-y^\delta\|^{p-1}\\
&\leq\frac{\delta^{p-1}}{\alpha}+4\,p\,\frac{D\left(\Phi^{-1}(\delta)\right)}{\delta}\\
&\leq\left(\frac{1}{c_1}+4\,p\right)\,\frac{D\left(\Phi^{-1}(\delta)\right)}{\delta}\\
&=\left(\frac{1}{c_1}+4\,p\right)\,\Phi^{-1}(\delta).
\end{align*}
This proves the theorem with $c=\frac{1}{c_1}+4\,p$.
\end{proof}

Theorem~\ref{th:rate} and Lemma~\ref{th:Dvsc} yield the convergence rate
\begin{equation*}
B^\Omega_{\xi_\alpha^\delta}(x^\dagger,x_\alpha^\delta)=\cO\left(D\left(\Phi^{-1}(\delta)\right)\right),\qquad\delta\to 0.
\end{equation*}
From Theorem~\ref{th:converse} we obtain
\begin{equation*}
D\left(c\,\Phi^{-1}(\delta)\right)\leq B^\Omega_{\xi_\alpha^\delta}(x^\dagger,x_\alpha^\delta)\qquad\text{for all $\delta>0$}
\end{equation*}
with some constant $c$ using the same parameter choice as in Theorem~\ref{th:rate}.
The constant $c$ in the lower bound can be removed in many situations.

\begin{remark}
From Theorem~\ref{th:converse} we obtain the desired lower bound
\begin{equation*}
\tilde{c}\,D\left(\Phi^{-1}(\delta)\right)\leq B^\Omega_{\xi_\alpha^\delta}(x^\dagger,x_\alpha^\delta)\qquad\text{for all $\delta>0$}
\end{equation*}
with some $\tilde{c}>0$ if $D$ does not decay too fast, that is, if
\begin{equation*}
\tilde{c}\,D\left(\Phi^{-1}(\delta)\right)\leq D\left(c\,\Phi^{-1}(\delta)\right)
\end{equation*}
with $c$ from the theorem.
Examples for the existence of such $\tilde{c}$ are the cases
\begin{equation*}
D(r)\sim r^{-a}\qquad\text{and}\qquad D(r)\sim(\ln r)^{-a}
\end{equation*}
with $a>0$.
A counter example is $D(r)\sim \exp(-r)$.
\end{remark}

\section{Concluding remarks}

Optimality of convergence rates has been considered for ill-posed linear problems in Hilbert spaces several years ago.
The first result in this direction was \cite{Neu97} based on source conditions, followed by \cite{FleHofMat11} emphasizing the tight connections between convergence rates and distance functions of approximate source conditions.
In \cite[Chapter~13]{Fle12} converse results for variational source conditions in Hilbert spaces were presented, partly rediscovered and partly extended by \cite{AlbElbDehSch16,HohWei17}.

Theorem~\ref{th:converse} is, to the author's best knowledge, the first converse result for convergence rates in Banach spaces.
The proof technique is very similar to the one used in \cite{FleHofMat11} and \cite[Section~13.4]{Fle12}.
The important idea is how to ged rid of spectral theory.
Looking at the proof of Theorem~\ref{th:converse} we see that the second part, that is, estimating $\|\eta_\alpha^\delta\|$ does not require linearity of $A$ and, thus, is applicable to nonlinear inverse problems, too.

Restriction to linear mappings stems from using skewed Bregman distances. As already noted in Section~\ref{sc:rates}, for nonlinear mappings we cannot ensure that the required subgradient in the definition of the skewed Bregman distance exists.
For the some reason it is pointless to search for variational source conditions with a skewed Bregman distance as error measure on the left-hand side. This would only be possible if we could guarantee $\partial\Omega(x)\neq\emptyset$ for all $x\in X$.

Despite the mentioned drawbacks of skewed Bregman distances, they are, up to now, the only error measures allowing for converse results in Banach spaces. This odd situation either is a fortunate coincidence or the starting point for a new understanding of variational source conditions.
The author favors the second position, because one can prove the intimate connection
\begin{equation*}
D(\|\eta_\alpha^0\|)=B^\Omega_{\xi_\alpha^0}(x^\dagger,x_\alpha^0)
\end{equation*}
between variational source conditions and skewed Bregman distances.

Variational source conditions proved to be a very powerful tool for convergence rate analysis of linear and nonlinear inverse problems in Banach space. But up to now the following important questions have not been answered:
\begin{itemize}
\item
How to obtain higher-order rates from variational source conditions? Classical source conditions in Hilbert spaces yield convergence rates for Tikhonov regularization up to $\|x_\alpha^\delta-x^\dagger\|=\cO(\delta^{\frac{2}{3}})$. Via variational source conditions only rates up to $\|x_\alpha^\delta-x^\dagger\|=\cO(\sqrt{\delta})$ can be obtained. For linear problems in Hilbert or Banach spaces workarounds were suggested in \cite[Section~13.1.3]{Fle12} and \cite{Gra13,SprHoh17}, but those approaches do not allow generalization to nonlinear mappings.
\item
How to handle oversmoothing? Situations in which the regularized solutions are smoother than the exact solutions, that is, $\Omega(x_\alpha^\delta)<\infty$, but $\Omega(x^\dagger)=\infty$, cannot be handled by variational source conditions up to now.
We refer to \cite{HofMat17} for details and references.
\item
How to prove converse results with different error measures? As mentioned above, there are no converse results for convergence rates in Banach spaces and it is unclear how to find such results for variational source conditions with norms or non-skewed Bregman distances as error measure and with nonlinear mappings.
\end{itemize}


\bibliography{vsc}

\begin{thebibliography}{10}

\bibitem{AlbElbDehSch16}
V.~Albani, P.~Elbau, M.~V. de~Hoop, and O.~Scherzer.
\newblock {Optimal convergence rates results for linear inverse problems in
  Hilbert spaces}.
\newblock {\em Numerical Functional Analysis and Optimization}, 37(5):521--540,
  2016.

\bibitem{BotHof10}
R.~I. Bo{\c{t}} and B.~Hofmann.
\newblock {An extension of the variational inequality approach for nonlinear
  ill-posed problems}.
\newblock {\em Journal of Integral Equations and Applications}, 22(3):369--392,
  2010.

\bibitem{BueFleHof16}
S.~B{\"u}rger, J.~Flemming, and B.~Hofmann.
\newblock {On complex-valued deautoconvolution of compactly supported functions
  with sparse Fourier representation}.
\newblock {\em Inverse Problems}, 32(10):104006 (12pp), 2016.

\bibitem{EngHanNeu96}
H.~W. Engl, M.~Hanke, and A.~Neubauer.
\newblock {\em {Regularization of Inverse Problems}}.
\newblock Mathematics and Its Applications. Kluwer Academic Publishers,
  Dordrecht, 1996.

\bibitem{Fle12}
J.~Flemming.
\newblock {\em {Generalized Tikhonov regularization and modern convergence rate
  theory in Banach spaces}}.
\newblock Shaker Verlag, Aachen, 2012.

\bibitem{Fle17}
J.~Flemming.
\newblock {Existence of variational source conditions for nonlinear inverse
  problems in Banach spaces}.
\newblock {\em Journal of Inverse and Ill-Posed Problems}, 2017.
\newblock Ahead of print, available online, DOI: 10.1515/jiip-2017-0092.

\bibitem{FleGer17}
J.~Flemming and D.~Gerth.
\newblock {Injectivity and weak*-to-weak continuity suffice for convergence
  rates in $\ell^1$-regularization}.
\newblock {\em Journal of Inverse and Ill-Posed Problems}, 2017.
\newblock Ahead of print, available online, DOI: 10.1515/jiip-2017-0008.

\bibitem{FleHofMat11}
J.~Flemming, B.~Hofmann, and P.~Math{\'e}.
\newblock {Sharp converse results for the regularization error using distance
  functions}.
\newblock {\em Inverse Problems}, 27(2):025006 (18pp), 2011.

\bibitem{Gra10b}
M.~Grasmair.
\newblock {Generalized Bregman distances and convergence rates for non-convex
  regularization methods}.
\newblock {\em Inverse Problems}, 26(11):115014 (16pp), 2010.

\bibitem{Gra13}
M.~Grasmair.
\newblock {Variational inequalities and higher order convergence rates for
  Tikhonov regularisation on Banach spaces}.
\newblock {\em Journal of Inverse and Ill-Posed Problems}, 21(3):379--394,
  2013.

\bibitem{HofKalPoeSch07}
B.~Hofmann, B.~Kaltenbacher, C.~P{\"o}schl, and O.~Scherzer.
\newblock {A convergence rates result for Tikhonov regularization in Banach
  spaces with non-smooth operators}.
\newblock {\em Inverse Problems}, 23(3):987--1010, 2007.

\bibitem{HofMat12}
B.~Hofmann and P.~Math{\'e}.
\newblock {Parameter choice in Banach space regularization under variational
  inequalities}.
\newblock {\em Inverse Problems}, 28:104006 (17pp), 2012.

\bibitem{HofMat17}
B.~Hofmann and P.~Math{\'e}.
\newblock {Tikhonov regularization with oversmoothing penalty for non-linear
  ill-posed problems in Hilbert scales}.
\newblock {\em arXiv.org}, arXiv:1705.03289 [math.NA], November 2017.
\newblock https://arxiv.org/abs/1705.03289.

\bibitem{HohWei15}
T.~Hohage and F.~Weidling.
\newblock Verification of a variational source condition for acoustic inverse
  medium scattering problems.
\newblock {\em Inverse Problems}, 31(7):075006 (14pp), 2015.

\bibitem{HohWei17}
T.~Hohage and F.~Weidling.
\newblock Characterizations of variational source conditions, converse results,
  and maxisets of spectral regularization methods.
\newblock {\em SIAM Journal on Numerical Analysis}, 55(2):598--620, 2017.

\bibitem{HohWer13}
T.~Hohage and F.~Werner.
\newblock {Iteratively regularized Newton methods with general data misfit
  functionals and applications to Poisson data}.
\newblock {\em Numerische Mathematik}, 123(4):745--779, 2013.

\bibitem{Kin16}
S.~Kindermann.
\newblock {Convex Tikhonov regularization in Banach spaces: New results on
  convergence rates}.
\newblock {\em Journal of Inverse and Ill-Posed Problems}, 24(3):341--350,
  2016.

\bibitem{KoeHohWer16}
C.~K{\"o}nig, T.~Hohage, and F.~Werner.
\newblock {Convergence rates for exponentially ill-posed inverse problems with
  impulsive noise}.
\newblock {\em SIAM Journal on Numerical Analysis}, 54(1):341--360, 2016.

\bibitem{MatPer03}
P.~Math{\'{e}} and S.~V. Pereverzev.
\newblock {Geometry of linear ill-posed problems in variable Hilbert scales}.
\newblock {\em Inverse Problems}, 19(3):789--803, 2003.

\bibitem{Neu97}
A.~Neubauer.
\newblock On converse and saturation results for {T}ikhonov regularization of
  linear ill-posed problems.
\newblock {\em SIAM Journal on Numerical Analysis}, 34(2):517--527, 1997.

\bibitem{SchGraGroHalLen09}
O.~Scherzer, M.~Grasmair, H.~Grossauer, M.~Haltmeier, and F.~Lenzen.
\newblock {\em {Variational Methods in Imaging}}.
\newblock Number 167 in Applied Mathematical Sciences. Springer, New York,
  2009.

\bibitem{SchKalHofKaz12}
T.~Schuster, B.~Kaltenbacher, B.~Hofmann, and K.~S. Kazimierski.
\newblock {\em {Regularization Methods in Banach Spaces}}, volume~10 of {\em
  Radon Series on Computational and Applied Mathematics}.
\newblock De Gruyter, Berlin/Boston, 2012.

\bibitem{SprHoh17}
B.~Sprung and T.~Hohage.
\newblock {Higher order convergence rates for Bregman iterated variational
  regularization of inverse problems}.
\newblock {\em arXiv.org}, arXiv:1710.09244 [math.NA], October 2017.
\newblock https://arxiv.org/abs/1710.09244.

\end{thebibliography}

\end{document}